\documentclass{amsart}

\usepackage{amssymb,xcolor}

\newtheorem{theorem}{Theorem}[section]
\newtheorem{lemma}[theorem]{Lemma}
\newtheorem{corollary}[theorem]{Corollary}

\theoremstyle{definition}

\theoremstyle{remark}
\newtheorem{remark}[theorem]{Remark}

\numberwithin{equation}{section}

\newcommand{\rd}{\, \mathrm{d}}

\newcommand{\bsk}{\boldsymbol{k}}

\newcommand{\bszero}{\boldsymbol{0}}
\newcommand{\bsone}{\boldsymbol{1}}

\newcommand{\bsx}{\boldsymbol{x}}
\newcommand{\bsy}{\boldsymbol{y}}

\newcommand{\comp}{\mathrm{comp}}

\newcommand{\wor}{\mathrm{wor}}

\newcommand{\abs}{\mathrm{abs}}
\newcommand{\nor}{\mathrm{nor}}

\newcommand{\NN}{\mathbb{N}}
\newcommand{\PP}{\mathbb{P}}
\newcommand{\RR}{\mathbb{R}}

\newcommand{\ZZ}{\mathbb{Z}}

\DeclareMathOperator{\supp}{supp}

\begin{document}

\title[Polynomial tractability for convergent Fourier series]{Polynomial tractability for integration in an unweighted function space with absolutely convergent Fourier series}

\author{Takashi Goda}
\address{School of Engineering, University of Tokyo, 7-3-1 Hongo, Bunkyo-ku, 113-8656 Tokyo, Japan.}
\email{goda@frcer.t.u-tokyo.ac.jp}
\thanks{The work of the author is supported by JSPS KAKENHI Grant Number 20K03744.}


\date{}

\dedicatory{}

\commby{}

\begin{abstract}
In this note, we prove that the following function space with absolutely convergent Fourier series
\[ F_d:=\left\{ f\in L^2([0,1)^d)\:\middle| \: \|f\|:=\sum_{\bsk\in \ZZ^d}|\hat{f}(\bsk)|\max\left(1,\min_{j\in \supp(\bsk)}\log |k_j|\right)<\infty \right\}\]
with $\hat{f}(\bsk)$ being the $\bsk$-th Fourier coefficient of $f$ and $\supp(\bsk):=\{j\in \{1,\ldots,d\}\mid k_j\neq 0\}$ is polynomially tractable for multivariate integration in the worst-case setting. Here polynomial tractability means that the minimum number of function evaluations required to make the worst-case error less than or equal to a tolerance $\varepsilon$ grows only polynomially with respect to $\varepsilon^{-1}$ and $d$. It is important to remark that the function space $F_d$ is \emph{unweighted}, that is, all variables contribute equally to the norm of functions. Our tractability result is in contrast to those for most of the unweighted integration problems studied in the literature, in which polynomial tractability does not hold and the problem suffers from the curse of dimensionality. Our proof is constructive in the sense that we provide an explicit quasi-Monte Carlo rule that attains a desired worst-case error bound.
\end{abstract}

\maketitle

\section{Introduction}
We study numerical integration problems for multivariate functions defined over the $d$-dimensional unit cube. We denote the integral of a Riemann integrable function $f: [0,1)^d\to \RR$ by
\[ I_d(f):=\int_{[0,1)^d}f(\bsx)\rd \bsx. \]
We approximate $I_d(f)$ by a linear cubature algorithm
\[ Q_{d,N}(f):=\sum_{n=1}^{N}w_n f(\bsx_n),\]
with given sets of nodes $(\bsx_n)_{n=1,\ldots,N}\subset [0,1)^d$ and weights $(w_n)_{n=1,\ldots,N}\in \RR^N$. If all the weights $w_n$ are equal to $1/N$, the algorithm $Q_{d,N}$ is particularly called a quasi-Monte Carlo (QMC) rule \cite{DKS13}.

For a Banach space $F$ with norm $\|\cdot\|$, the initial error is defined by
\[ e^{\wor}(F,\emptyset):=\sup_{\substack{f\in F\\ \|f\|\leq 1}}\left| I_d(f)\right|, \]
and the worst-case error of $Q_{d,N}$ is by
\[ e^{\wor}(F,Q_{d,N}):=\sup_{\substack{f\in F\\ \|f\|\leq 1}}\left| I_d(f)-Q_{d,N}(f)\right|.\]
For a tolerance $\varepsilon \in (0,1)$, we denote the minimum number of function evaluations $N$ (among all possible $Q_{d,N}$) required to make the worst-case error less than or equal to $\varepsilon e^{\wor}(F,\emptyset)$ by
\[ N^{\nor}(\varepsilon,d):=\inf\left\{ N\in \NN\: \middle|\: \exists Q_{d,N}: e^{\wor}(F,Q_{d,N})\leq \varepsilon e^{\wor}(F,\emptyset)\right\}. \]
In the context of information-based complexity \cite{NW08,NW10,NW12,TWW88}, this quantity $N^{\nor}(\varepsilon,d)$ is called the \emph{information complexity} for the normalized error criterion and we are interested in the dependence of $N^{\nor}(\varepsilon,d)$ on $\varepsilon^{-1}$ and $d$. Here, if we define
\[ N^{\abs}(\varepsilon,d):=\inf\left\{ N\in \NN\: \middle|\: \exists Q_{d,N}: e^{\wor}(F,Q_{d,N})\leq \varepsilon \right\}, \]
we speak of the information complexity for the absolute error criterion. These two definitions coincide when $e^{\wor}(F,\emptyset)=1$, in which case we omit the superscripts and simply write $N(\varepsilon,d)$.

Let us give some typical notions of tractability depending on how $N(\varepsilon,d)$ (or, either $N^{\nor}(\varepsilon,d)$ or $N^{\abs}(\varepsilon,d)$) grows with respect to $\varepsilon^{-1}$ and $d$.
\begin{itemize}
    \item We say that the multivariate integration problem in $F$ is \emph{intractable,} if $N(\varepsilon,d)$ depends exponentially on either $\varepsilon^{-1}$ or $d$. 
    \item We say that the multivariate integration problem in $F$ is \emph{polynomially tractable,} if there exist constants $C,\alpha,\beta>0$ such that
    \[ N(\varepsilon,d)\leq C\varepsilon^{-\alpha}d^{\beta},\]
    for all dimensions $d\in \NN$ and $\varepsilon\in (0,1)$.
    \item We say that the multivariate integration problem in $F$ is \emph{strongly polynomially tractable,} if there exist constants $C,\alpha>0$ such that
    \[ N(\varepsilon,d)\leq C\varepsilon^{-\alpha},\]
    for all dimensions $d\in \NN$ and $\varepsilon\in (0,1)$.
\end{itemize}

As pointed out in \cite{GW11}, for most of the \emph{unweighted} multivariate integration problems studied so far, that is, the problems in which all variables or groups of variables play the same role, polynomial tractability does \emph{not} hold and the problem suffers from the curse of dimensionality; see  \cite{HNUW14a,HNUW14b,HNW11} for such examples. Moreover, as one of the examples closely related to this note, an intractability result was proven in \cite{SW97}: for any $\alpha>1$, the integration in the unweighted function space defined by
\[ E_{d.\alpha}:=\left\{ f: [0,1)^d\to \RR\:\middle| \: \|f\|:=\sup_{\bsk\in \ZZ^d}|\hat{f}(\bsk)|\prod_{j=1}^{d}\max(1,|k_j|^{\alpha})<\infty \right\},\]
where 
\[ \hat{f}(\bsk)=\int_{[0,1)^d}f(\bsx)\exp(-2\pi i \bsk\cdot \bsx)\rd \bsx\] 
with $\cdot$ being the inner product denotes the $\bsk$-th Fourier coefficient of $f$ was proven to be intractable by showing that  $\inf_{Q_{d,N}}e^{\wor}(E_{d.\alpha},Q_{d,N})=1$ if $N< 2^d$.

Fortunately, it is not necessarily the case that the integration in any unweighted function space is not polynomially tractable. An important example was found by a seminal work of Heinrich, Novak, Wasilkowski and Wo\'{z}niakowski \cite{HNWW00}, who revealed that the unweighted Sobolev space with first-order dominating mixed smoothness is indeed polynomially tractable by proving that the inverse of the star discrepancy depends only linearly on $d$. As another example, which is more relevant to this note, Dick proved in \cite{D14} that, for $0<\alpha\leq 1$ and $1\leq p\leq \infty$, the integration in the following unweighted function space is polynomially tractable:
\[ K_{d,\alpha,p} = \left\{ f\in L^2([0,1)^d)\:\middle| \:   \|f\|:=\sum_{\bsk\in \ZZ^d}|\hat{f}(\bsk)|+|f|_{H_{\alpha,p}}<\infty \right\}, \]
where $|f|_{H_{\alpha,p}}$ denotes the H\"{o}lder semi-norm
\[ |f|_{H_{\alpha,p}}=\sup_{\bsx,\bsy,\bsx+\bsy\in [0,1)^d}\frac{|f(\bsx+\bsy)-f(\bsx)|}{\|\bsy\|_{\ell_p}^{\alpha}},\]
with $\|\cdot\|_{\ell_p}$ being the $\ell^p$ norm. His proof is constructive, i.e., he gave an explicit QMC rule based on Korobov's $p$-sets which achieves a desired worst-case error bound. We refer to \cite{DGPW17,DP15} for more relevant results.

In this note, we introduce the following unweighted function space
\[ F_d:=\left\{ f\in L^2([0,1)^d)\:\middle| \: \|f\|:=\sum_{\bsk\in \ZZ^d}|\hat{f}(\bsk)| \max\left(1, \min_{j\in \supp(\bsk)}\log |k_j|\right) <\infty \right\}\]
with $\supp(\bsk):=\{j\in \{1,\ldots,d\}\mid k_j\neq 0\}$. In comparison to the space $K_{d,\alpha,p}$, we do not impose the H\"{o}lder continuity of $f$, and instead, assume a slightly faster decay of the Fourier coefficients. As an example, let us consider the function 
\[ \omega_{b,\delta}(x)=\sum_{k=1}^{\infty}\frac{\cos(2\pi b^kx)}{k^{2+\delta}}, \]
for $b\in \NN$ with $b\geq 2$ and $\delta\geq 0$. In \cite[Section~4.1]{H16}, Hardy originally claimed that\footnote{Even more, he claimed $\omega_{b,0}(x+h)-\omega_{b,0}(x)\neq o(|\log|h||^{-2})$, and gave another example $f$ with absolutely convergent Fourier series and $f(x+h)-f(x)\neq O(|\log|h||^{-1})$.}, for any $b$, the function $\omega_{b,\delta}$ with $\delta=0$ is not H\"{o}lder continuous of any order $0<\alpha\leq 1$. In fact, his argument made for the Weierstrass function can be also applied to the case $\delta>0$ to see the non-H\"{o}lder continuity of $\omega_{b,\delta}$ for any $b$ and $\delta$. Therefore, the function $\omega_{b,\delta}$ does not belong to $K_{d,\alpha,p}$ for any $0<\alpha\leq 1$, whereas it belongs to $F_d$ if $\delta>0$.

The main result of this note is to prove that the integration in $F_d$ is polynomially tractable. Similarly to \cite{D14}, our proof is constructive. The QMC rule we give is inspired by the use of Korobov's $p$-sets in \cite{D14,DP15} and the use of different primes for $p$ in \cite{DGS22,KKNU19}. Although the prime $p$ is randomly chosen in the randomized algorithms in \cite{DGS22,KKNU19}, we consider the composition of Korobov's $p$-sets with different primes $p$, which suits our analysis for the worst-case setting in this note.

\section{The results}
Throughout this note, we denote the set of integers by $\ZZ$ and the set of positive integers by $\NN$. For an integer $M\geq 2$, we write
\[ \PP_M := \left\{ \lceil M/2\rceil <p\leq M\: \middle|\: \text{$p$ is prime}\right\}.\]
It is known from \cite[Corollaries~1--3]{RS62} that there exist absolute constants $0<c_{\PP}<1$ and $C_{\PP}>1$ such that
\begin{align}\label{eq:card_PPM}
    c_{\PP}\frac{M}{\log M}\leq |\PP_M|\leq  C_{\PP}\frac{M}{\log M}
\end{align}
holds for any $M\geq 2$. Moreover,  we denote the fractional part of a non-negative real number $x$ by $\{x\}=x-\lfloor x\rfloor$. 

\subsection{Composite Korobov's $p$-sets}
Following \cite[Section~4.3]{HW81} and \cite[Section~2]{DP15}, let us define the following three point sets in $[0,1)^d$ for a prime $p$, which are called Korobov's $p$-sets.
\begin{itemize}
\item Define $S_{p,d}=\{\bsx_n^{(p)}\mid 0\leq n<p\}$ with
\[ \bsx_n^{(p)}=\left( \left\{ \frac{n}{p}\right\}, \left\{ \frac{n^2}{p}\right\},\ldots, \left\{ \frac{n^{d}}{p}\right\}\right).\]
\item Define $T_{p,d}=\{\tilde{\bsx}_n^{(p)}\mid 0\leq n<p^2\}$ with
\[ \tilde{\bsx}_n^{(p)}=\left( \left\{ \frac{n}{p^2}\right\}, \left\{ \frac{n^2}{p^2}\right\},\ldots, \left\{ \frac{n^{d}}{p^2}\right\}\right).\]
\item Define $U_{p,d}=\{\bsx_{k,n}^{(p)}\mid 0\leq k,n<p\}$ with
\[ \bsx_{k,n}^{(p)}=\left( \left\{ \frac{kn}{p}\right\}, \left\{ \frac{kn^2}{p}\right\},\ldots, \left\{ \frac{kn^{d}}{p}\right\}\right).\]
\end{itemize}
It is obvious that $|S_{p,d}|=p$ and $|T_{p,d}|=|U_{p,d}|=p^2$.

Now, for a given $M\geq 2$, we take the composition (or multi-set union) of $S_{p,d}$'s, $T_{p,d}$'s and $U_{p,d}$'s over all $p\in \PP_M$, respectively, to introduce
\[ S_{M,d}^{\comp} = \bigcup_{p\in \PP_M}S_{p,d},\quad T_{M,d}^{\comp} = \bigcup_{p\in \PP_M}T_{p,d},\quad \text{and}\quad U_{M,d}^{\comp} = \bigcup_{p\in \PP_M}U_{p,d}.\] 
The cardinality of each composite point set is trivially equal to
\[ |S_{M,d}^{\comp}|=\sum_{p\in \PP_M}p,\quad |T_{M,d}^{\comp}|=|U_{M,d}^{\comp}|=\sum_{p\in \PP_M}p^2,\]
respectively, which can be bounded from below and above by using \eqref{eq:card_PPM} and $M/2<p\leq M$ for all $p\in \PP_M$. With abuse of notation, the cubature algorithms we consider below are QMC rules:
\begin{align*}
    & Q_{S_{M,d}^{\comp}}(f):=\frac{1}{|S_{M,d}^{\comp}|}\sum_{p\in \PP_M}\sum_{n=0}^{p-1}f(\bsx_n^{(p)}), \\
    & Q_{T_{M,d}^{\comp}}(f):=\frac{1}{|T_{M,d}^{\comp}|}\sum_{p\in \PP_M}\sum_{n=0}^{p^2-1}f(\tilde{\bsx}_n^{(p)}),\quad \text{and}\\
    & Q_{U_{M,d}^{\comp}}(f):=\frac{1}{|U_{M,d}^{\comp}|}\sum_{p\in \PP_M}\sum_{k,n=0}^{p-1}f(\bsx_{k,n}^{(p)}).
\end{align*}

Before moving on to the worst-case error analysis in $F_d$, we give some auxiliary results on the exponential sums. First, the following result, which traces its history back to at least the work of Weil in \cite{W48}, was shown in \cite[Lemmas~4.3--4.5]{DP15}; see also \cite[Lemmas~4.5--4.7]{HW81}.

\begin{lemma}\label{lem:exp_sum}
    Let $p$ be a prime. For any $\bsk\in \ZZ^d\setminus \{\bszero\}$ such that there exists at least one index $j\in \{1,\ldots,d\}$ such that $k_j$ is not divided by $p$, i.e., $p\nmid \bsk$, we have
    \begin{align*}
        & \left|\frac{1}{p}\sum_{n=0}^{p-1}\exp\left( 2\pi i\bsk\cdot \bsx_n^{(p)}\right)\right| \leq \frac{d-1}{\sqrt{p}}, \\
        & \left|\frac{1}{p^2}\sum_{n=0}^{p^2-1}\exp\left( 2\pi i\bsk\cdot \tilde{\bsx}_n^{(p)}\right)\right| \leq \frac{d-1}{p}, \quad \text{and}\\
        & \left|\frac{1}{p^2}\sum_{k,n=0}^{p^2-1}\exp\left( 2\pi i\bsk\cdot \bsx_{k,n}^{(p)}\right)\right| \leq \frac{d-1}{p}.
    \end{align*}
\end{lemma}

\begin{remark}\label{rem:exp_sum}
    If the condition in Lemma~\ref{lem:exp_sum} does not hold, i.e., if all the components of $\bsk$ are divisible by $p$, we have a trivial bound on the respective exponential sum, which is 1.
\end{remark}

Using the above result, we obtain bounds on the exponential sums for the composite point sets as follows.
\begin{corollary}\label{cor:composite_exp_sum}
    Let $M\geq 2$. For any $\bsk\in \ZZ^d\setminus \{\bszero\}$, we have
    \begin{align*}
        & \left|\frac{1}{|S_{M,d}^{\comp}|}\sum_{p\in \PP_M}\sum_{n=0}^{p-1}\exp\left( 2\pi i\bsk\cdot \bsx_n^{(p)}\right)\right| \leq \frac{2(d-1)}{\sqrt{M}}+\frac{4 \min_{j\in \supp(\bsk)}\log |k_j|}{c_{\PP}M}, \\
        & \left|\frac{1}{|T_{M,d}^{\comp}|}\sum_{p\in \PP_M}\sum_{n=0}^{p^2-1}\exp\left( 2\pi i\bsk\cdot \tilde{\bsx}_n^{(p)}\right)\right| \leq \frac{4(d-1)}{M}+\frac{4\min_{j\in \supp(\bsk)}\log |k_j|}{c_{\PP}M}, \quad \text{and}\\
        & \left|\frac{1}{|U_{M,d}^{\comp}|}\sum_{p\in \PP_M}\sum_{k,n=0}^{p-1}\exp\left( 2\pi i\bsk\cdot \bsx_{k,n}^{(p)}\right)\right| \leq \frac{4(d-1)}{M}+\frac{4\min_{j\in \supp(\bsk)}\log |k_j|}{c_{\PP}M}.
    \end{align*}
\end{corollary}
\begin{proof}
Since each inequality can be shown in a similar way, we only give a proof for the first inequality. Using the triangle inequality, Lemma~\ref{lem:exp_sum} and Remark~\ref{rem:exp_sum}, it holds that
\begin{align*}
    \left|\frac{1}{|S_{M,d}^{\comp}|}\sum_{p\in \PP_M}\sum_{n=0}^{p-1}\exp\left( 2\pi i\bsk\cdot \bsx_n^{(p)}\right)\right| & \leq \frac{1}{|S_{M,d}^{\comp}|}\sum_{p\in \PP_M}\left|\sum_{n=0}^{p-1}\exp\left( 2\pi i\bsk\cdot \bsx_n^{(p)}\right)\right| \\
    & \leq \frac{1}{|S_{M,d}^{\comp}|}\sum_{\substack{p\in \PP_M\\ p\nmid \bsk}}(d-1)\sqrt{p}+\frac{1}{|S_{M,d}^{\comp}|}\sum_{\substack{p\in \PP_M\\ p\mid \bsk}}p.
\end{align*}
Regarding the first term on the rightmost side above, 
the sum over $p\in \PP_M$ is bounded from above by $|\PP_M|(d-1)\sqrt{M}$, whereas $|S_{M,d}^{\comp}|$ is bounded from below by $|\PP_M|(M/2)$. Thus we have
\[ \frac{1}{|S_{M,d}^{\comp}|}\sum_{\substack{p\in \PP_M\\ p\nmid \bsk}}(d-1)\sqrt{p}\leq \frac{2(d-1)}{\sqrt{M}}.\]
For the second term, by using the fact that, for any integers $k,n\in \NN$, $k$ has at most $\log_n k$ prime divisors greater than $n$, we have
\begin{align*}
    \frac{1}{|S_{M,d}^{\comp}|}\sum_{\substack{p\in \PP_M\\ p\mid \bsk}}p & \leq \frac{M}{|S_{M,d}^{\comp}|}\sum_{\substack{p\in \PP_M\\ p\mid \bsk}}1  = \frac{M}{|S_{M,d}^{\comp}|}\sum_{p\in \PP_M}\prod_{j\in \supp(\bsk)}\bsone_{p\mid k_j}\\
    & \leq \frac{M}{|S_{M,d}^{\comp}|}\min_{j\in \supp(\bsk)}\sum_{p\in \PP_M}\bsone_{p\mid k_j}\leq \frac{M}{|S_{M,d}^{\comp}|}\min_{j\in \supp(\bsk)}\log_{\lceil M/2\rceil+1}|k_j|,
\end{align*}
where $\bsone_{A}$ denotes the indicator function of an event $A$. As $|S_{M,d}^{\comp}|$ is bounded from below by $|\PP_M|(M/2)$, which itself is further bounded from below by using \eqref{eq:card_PPM}, we obtain
\begin{align*}
    \frac{1}{|S_{M,d}^{\comp}|}\sum_{\substack{p\in \PP_M\\ p\mid \bsk}}p \leq \frac{2\log M}{c_{\PP}M}\cdot \frac{2\min_{j\in \supp(\bsk)}\log |k_j|}{\log M}=\frac{4\min_{j\in \supp(\bsk)}\log |k_j|}{c_{\PP}M},
\end{align*}
which implies the result.
\end{proof}

\begin{remark}
    By considering the composition of Korobov's $p$-sets with different primes $p$, we obtain bounds on the exponential sums without any need to distinguish the cases whether $p\mid \bsk$ or $p\nmid \bsk$, which is crucial in the worst-case error analysis in the space $F_d$. In comparison, in \cite{D14}, the case $p\mid \bsk$ is dealt with the H\"{o}lder continuity and the case $p\nmid \bsk$ is with the absolute convergence of Fourier series. This is how the worst-case error of (non-composite) Korobov's $p$-sets in the space $K_{d,\alpha,p}$ is analyzed.
\end{remark}

\subsection{Worst-case error bound and tractability}
We now prove upper bounds on the worst-case error of our QMC rules in $F_d$.
\begin{theorem}
    Let $M\geq 2$. The worst-case error of QMC rules $Q_{S_{M,d}^{\comp}}$, $Q_{T_{M,d}^{\comp}}$ and $Q_{U_{M,d}^{\comp}}$ in $F_d$ are bounded from above by
    \begin{align*}
        & e^{\wor}(F_d,Q_{S_{M,d}^{\comp}})\leq \frac{8d}{c_{\PP}|S_{M,d}^{\comp}|^{1/4}},\\
        & e^{\wor}(F_d,Q_{T_{M,d}^{\comp}})\leq \frac{8d}{c_{\PP}|T_{M,d}^{\comp}|^{1/3}},\quad \text{and}\\
        & e^{\wor}(F_d,Q_{U_{M,d}^{\comp}})\leq \frac{8d}{c_{\PP}|U_{M,d}^{\comp}|^{1/3}},
    \end{align*}
    respectively.
\end{theorem}
\begin{proof}
As in Corollary~\ref{cor:composite_exp_sum}, we only give a proof for the first worst-case error bound, as the remaining two bounds can be shown similarly.
For $f\in F_d$, by expanding it into the Fourier series and then applying the triangle inequality and Corollary~\ref{cor:composite_exp_sum}, we have
\begin{align*}
    & \left|I_d(f)-\frac{1}{|S_{M,d}^{\comp}|}\sum_{p\in \PP_M}\sum_{n=0}^{p-1}f(\bsx_n^{(p)})\right| \\
    & \quad = \left| \sum_{\bsk\in \ZZ^d\setminus \{\bszero\}}\hat{f}(\bsk)\frac{1}{|S_{M,d}^{\comp}|}\sum_{p\in \PP_M}\sum_{n=0}^{p-1}\exp(2\pi i\bsk\cdot \bsx_n^{(p)})\right|\\
    & \quad \leq \sum_{\bsk\in \ZZ^d\setminus \{\bszero\}}| \hat{f}(\bsk)|\left|\frac{1}{|S_{M,d}^{\comp}|}\sum_{p\in \PP_M}\sum_{n=0}^{p-1}\exp(2\pi i\bsk\cdot \bsx_n^{(p)})\right|\\
    & \quad \leq \sum_{\bsk\in \ZZ^d\setminus \{\bszero\}}| \hat{f}(\bsk)|\left( \frac{2(d-1)}{\sqrt{M}}+\frac{4\min_{j\in \supp(\bsk)}\log |k_j|}{c_{\PP}M}\right) \\
    & \quad \leq \frac{8}{c_{\PP}\sqrt{M}}\sum_{\bsk\in \ZZ^d\setminus \{\bszero\}}| \hat{f}(\bsk)|\max\left( d-1,\min_{j\in \supp(\bsk)}\log |k_j|\right) \\
    & \quad \leq \frac{8d}{c_{\PP}\sqrt{M}}\sum_{\bsk\in \ZZ^d\setminus \{\bszero\}}| \hat{f}(\bsk)|\max\left(1,\min_{j\in \supp(\bsk)}\log |k_j|\right) \leq \frac{8d}{c_{\PP}\sqrt{M}}\|f\|.
\end{align*}
Thus the worst-case error is bounded by
\[ e^{\wor}(F_d,Q_{S_{M,d}^{\comp}})\leq \frac{8d}{c_{\PP}\sqrt{M}}.\]
As the cardinality of $S_{M,d}^{\comp}$ is trivially bounded from above by $M^2$, the worst-case error bound in the theorem follows.
\end{proof}

In terms of the total number of function evaluations, the convergence rate of $Q_{S_{M,d}^{\comp}}$ is slower than $Q_{T_{M,d}^{\comp}}$ and $Q_{U_{M,d}^{\comp}}$. By applying the latter two bounds on the worst-case error in $F_d$, the information complexity is bounded as follows.
\begin{corollary}
    For any tolerance $\epsilon\in (0,1)$ and dimension $d\in \NN$, the information complexity both for the normalized and absolute error criteria is bounded by
    \[ N(\varepsilon,d)\leq (8c_{\PP}^{-1})^3\varepsilon^{-3}d^3.\]
\end{corollary}
\begin{proof}
The initial error in $F_d$ is given by
\[ e^{\wor}(F_d,\emptyset)=\sup_{\substack{f\in F_d\\ \|f\|\leq 1}}| \hat{f}(\bszero)|=1,\]
where the supremum is attained when $\hat{f}(\bszero)=\pm 1$ and $\hat{f}(\bsk)=0$ for all $\bsk\in \ZZ^d\setminus \{\bszero\}$. Thus the information complexity for the normalized error criterion coincides with that for the absolute error criterion. To obtain an upper bound on $N(\varepsilon,d)$, it suffices to consider the inequality
\[ e^{\wor}(F_d,Q_{T_{M,d}^{\comp}})\leq \frac{8d}{c_{\PP}|T_{M,d}^{\comp}|^{1/3}}\leq \varepsilon. \]
The result of the corollary follows from the fact that $N(\varepsilon,d)$ is less than or equal to the minimum value of $|T_{M,d}^{\comp}|$ which satisfies this inequality.
\end{proof}

This way, we clearly see that the multivariate integration problem in $F_d$ is polynomially tractable. The corresponding $\varepsilon$- and $d$-exponents are less than or equal to 3. The following problems are open for future work.

\begin{itemize}
    \item Study whether the $\varepsilon$-  and $d$-exponents in $F_d$ can be improved.
    \item Examine the relation of the functions in $F_d$ to a class of continuous functions, such as log-H\"{o}lder continuous functions.
    \item Discuss the possibility to extend the result of this note to an unweighted function space with the norm
    \[ \|f\|:=\sum_{\bsk\in \ZZ^d}|\hat{f}(\bsk)| g\left(\min_{j\in \supp(\bsk)} |k_j|\right),\]
    for $g: \NN\cup\{0\}\to \RR_{>0}$ being a more slowly increasing function than the logarithmic function or even a constant function.
\end{itemize}

\bibliographystyle{amsplain}
\bibliography{ref.bib}

\providecommand{\bysame}{\leavevmode\hbox to3em{\hrulefill}\thinspace}
\providecommand{\MR}{\relax\ifhmode\unskip\space\fi MR }
\providecommand{\MRhref}[2]{%
  \href{http://www.ams.org/mathscinet-getitem?mr=#1}{#2}
}
\providecommand{\href}[2]{#2}
\begin{thebibliography}{10}

\bibitem{D14}
J.~Dick, \emph{Numerical integration of {H}\"{o}lder continuous, absolutely
  convergent {F}ourier, {F}ourier cosine, and {W}alsh series}, Journal of
  Approximation Theory \textbf{183} (2014), 14--30.

\bibitem{DGS22}
J.~Dick, T.~Goda, and K.~Suzuki, \emph{Component-by-component construction of
  randomized rank-1 lattice rules achieving almost the optimal randomized error
  rate}, Mathematics of Computation \textbf{91} (2022), 2771--2801.

\bibitem{DGPW17}
J.~Dick, D.~Gomez-Perez, F.~Pillichshammer, and A.~Winterhof, \emph{Digital
  inversive vectors can achieve polynomial tractability for the weighted star
  discrepancy and for multivariate integration}, Proceedings of the American
  Mathematical Society \textbf{145} (2017), 3297--3310.

\bibitem{DKS13}
J.~Dick, F.~Y. Kuo, and I.~H. Sloan, \emph{High-dimensional integration: the
  quasi-{M}onte {C}arlo way}, Acta Numerica \textbf{22} (2013), 133--288.

\bibitem{DP15}
J.~Dick and F.~Pillichshammer, \emph{The weighted star discrepancy of
  {K}orobov's $p$-sets}, Proceedings of the American Mathematical Society
  \textbf{143} (2015), 5043--5057.

\bibitem{GW11}
M.~Gnewuch and H.~Wo\'{z}niakowski, \emph{Quasi-polynomial tractability},
  Journal of Complexity \textbf{27} (2011), 312--330.

\bibitem{H16}
G.~H. Hardy, \emph{Weierstrass's non-differentiable function}, Transactions of
  the American Mathematical Society \textbf{17} (1916), no.~3, 301--325.

\bibitem{HNWW00}
S.~Heinrich, E.~Novak, G.~W. Wasilkowski, and H.~Wo\'{z}niakowski, \emph{The
  inverse of the star-discrepancy depends linearly on the dimension}, Acta
  Arithmetica \textbf{96} (2000), no.~3, 279--302.

\bibitem{HNUW14a}
A.~Hinrichs, E.~Novak, M.~Ullrich, and H.~Wo\'{z}niakowski, \emph{The curse of
  dimensionality for numerical integration of smooth functions}, Mathematics of
  Computation \textbf{83} (2014), no.~290, 2853--2863.

\bibitem{HNUW14b}
\bysame, \emph{The curse of dimensionality for numerical integration of smooth
  functions {II}}, Journal of Complexity \textbf{30} (2014), no.~2, 117--143.

\bibitem{HNW11}
A.~Hinrichs, E.~Novak, and H.~Wo\'{z}niakowski, \emph{The curse of
  dimensionality for the class of monotone functions and for the class of
  convex functions}, Journal of Approximation Theory \textbf{163} (2011),
  no.~8, 955--965.

\bibitem{HW81}
L.~K. Hua and Y.~Wang, \emph{{Applications of Number Theory to Numerical
  Analysis}}, Springer-Verlag, Berlin, 1981.

\bibitem{KKNU19}
P.~Kritzer, F.~Y. Kuo, D.~Nuyens, and Mario Ullrich, \emph{Lattice rules with
  random $n$ achieve nearly the optimal $\mathcal{O}(n^{-\alpha-1/2})$ error
  independently of the dimension}, Journal of Approximation Theory \textbf{240}
  (2019), 96--113.

\bibitem{NW08}
E.~Novak and H.~Wo\'{z}niakowski, \emph{Tractability of {M}ultivariate
  {P}roblems, {V}olume {I}: {L}inear {I}nformation}, EMS Press, Z\"{u}rich,
  2008.

\bibitem{NW10}
\bysame, \emph{Tractability of {M}ultivariate {P}roblems, {V}olume {II}:
  {S}tandard {I}nformation for {F}unctionals}, EMS Press, Z\"{u}rich, 2010.

\bibitem{NW12}
\bysame, \emph{Tractability of {M}ultivariate {P}roblems, {V}olume {III}:
  {S}tandard {I}nformation for {O}perators,}, EMS Press, Z\"{u}rich, 2012.

\bibitem{RS62}
J.~B. Rosser and L.~Schoenfeld, \emph{Approximate formulas for some functions
  of prime numbers}, Illinois Journal of Mathematics \textbf{6} (1962), 64--94.

\bibitem{SW97}
I.~H. Sloan and H.~Wo\'{z}niakowski, \emph{An intractability result for
  multiple integration}, Mathematics of Computation \textbf{66} (1997),
  1119--1124.

\bibitem{TWW88}
J.~F. Traub, G.~W. Wasilkowski, and H.~Wo\'{z}niakowski,
  \emph{Information-{B}ased {C}omplexity}, Academic Press, New York, 1988.

\bibitem{W48}
A.~Weil, \emph{On some exponential sums}, Proceedings of the National Academy
  of Sciences \textbf{34} (1948), no.~5, 204--207.

\end{thebibliography}

\end{document}